\documentclass[12pt,letterpaper]{article}

\usepackage[T1]{fontenc}
\usepackage[utf8]{inputenc}
\usepackage{lmodern}
\usepackage{amsmath,amssymb,amsthm,mathtools}
\usepackage[margin=1in]{geometry}
\usepackage{microtype}
\usepackage[numbers,sort&compress]{natbib}
\usepackage[hidelinks]{hyperref}

\allowdisplaybreaks
\numberwithin{equation}{section}
\newtheorem{theorem}{Theorem}[section]
\newtheorem{lemma}[theorem]{Lemma}

\theoremstyle{remark}

\newcommand{\R}{\mathbb{R}}
\newcommand{\cT}{\mathcal{T}}
\newcommand{\Emb}{\operatorname{Emb}}

\newcommand{\1}{\mathbf{1}}
\newcommand{\dotcuprel}{\mathbin{\dot\cup}}
\newcommand{\eps}{\varepsilon}

\title{On decomposition thresholds for odd cycles}
\author{%
Jianfeng Hou\thanks{Email: \texttt{jfhou@fzu.edu.cn}}\textsuperscript{1},
Caiyun Hu\thanks{Email: \texttt{hucaiyun.fzu@gmail.com}}\textsuperscript{1},
Hezhi Wang\thanks{Email: \texttt{wanghezhi128@gmail.com}}\textsuperscript{1}\\[1em]
{\normalsize\textsuperscript{1}Center for Discrete Mathematics, Fuzhou University, Fuzhou, China}%
}
\date{}

\makeatletter
\renewcommand{\@maketitle}{%
  \newpage
  \null
  \vskip -5.32pt%
  \begin{center}%
    \vspace*{-17.23pt}%
    {\LARGE\bfseries \@title \par}%
    \ifx\@author\@empty\else
      \vskip 1.5em%
      {\large
        \def\@makefnmark{\hbox{\@textsuperscript{\normalfont\@thefnmark}}}%
        \lineskip .5em%
        \begin{tabular}[t]{c}%
          \@author
        \end{tabular}\par}%
    \fi
    \ifx\@date\@empty\else
      \vskip 1em%
      {\large \@date}%
    \fi
  \end{center}%
  \par
  \vskip 12pt}
\makeatother

\begin{document}
\microtypesetup{activate=false}
\maketitle

\begin{abstract}
We determine the asymptotic minimum-degree threshold for decompositions into cycles of a fixed odd length. We prove that, for each odd integer $\ell\ge3$,  every $\eps>0$ and sufficiently large $n$, every  $n$-vertex  $C_\ell$-divisible graph $G$ with minimum degree 
$\delta(G)\ge (\ell/(2\ell-2)+\eps)n$ has a $C_\ell$-decomposition, where $C_{\ell}$ denotes the cycle of length $\ell$. This settles the remaining cases $\ell=5,7,\ldots,71$ and completes the determination of the thresholds for all odd cycles. Our main ingredient is a sharp fractional decomposition
theorem for weighted triangles: for every real $a\ge 1$,
every $n$-vertex graph $G$ with $\delta(G)\ge
\left(\frac12+\frac{1}{2(a+1)}\right)n$ 
admits a fractional decomposition into triangles
with edge weights $a,1,1$. 
\end{abstract}



\section{Introduction}\label{sec:introduction}

Let $F$ be a fixed graph. An \emph{$F$-decomposition} of a graph $G$ is a collection of edge-disjoint copies of $F$ whose edge sets partition $E(G)$. Clearly, such a decomposition requires $e(F)\mid e(G)$ and the greatest common divisor of the degrees of $F$ to divide the degree of every vertex of $G$. We call a graph satisfying these necessary conditions \emph{$F$-divisible}. 

A fundamental problem in graph decomposition theory is
to determine when the necessary divisibility conditions
are also sufficient.  By a classical theorem of Wilson~\cite{Wilson1975},
every sufficiently large $F$-divisible complete graph
admits an $F$-decomposition. For general graphs, however, divisibility alone is
far from sufficient, and determining which graphs
admit an $F$-decomposition is considerably more difficult.
A natural approach is to impose a minimum-degree
condition on the host graph.

For a fixed graph $F$ with at least one edge, its \emph{decomposition threshold} $\delta_F$ is the infimum of all $d\in[0,1]$ such that, for  $\eps>0$ and a sufficiently large integer $n$, every $n$-vertex  $F$-divisible graph $G$  with $\delta(G)\ge(d+\eps)n$ has an $F$-decomposition. The systematic study of these thresholds was initiated by Barber, K\"uhn, Lo and Osthus~\cite{barber2016}, who developed the iterative absorption method for decomposing dense graphs. Further general results were obtained by Glock, K\"uhn, Lo, Montgomery and Osthus~\cite{glock2019}. 
For further related results on graph decompositions, see, for example,~\cite{bryant2014,delcourt2021,glock2023,keevashdesigns}.
Even for simple $F$, however, determining $\delta_F$ can be difficult.

Cycle decompositions exhibit a notable distinction between even and odd lengths. For $\ell\ge 3$, let $C_{\ell}$ be a cycle of length $\ell$.  For even cycles, Barber, K\"uhn, Lo and Osthus~\cite{barber2016} proved that
\[
 \delta_{C_4}=\frac23
 \qquad\text{and}\qquad
 \delta_{C_\ell}=\frac12\quad\text{for every even }\ell\ge6.
\]
Taylor~\cite{taylor2019} subsequently obtained exact results in the even-cycle setting. The corresponding thresholds for odd cycles are governed by a different obstruction. Note that every cycle crosses a given cut an even number of times. In particular, an odd cycle of length $\ell$ contains at most $\ell-1$ crossing edges, which means that a $C_{\ell}$-decomposition must contain enough internal edges to accommodate the crossing edges. As observed in~\cite{barber2016,bryant2026}, this yields the lower bound
\begin{equation}\label{eq:odd-cycle-lower-intro}
 \delta_{C_\ell}\ge\frac12+\frac1{2\ell-2}
 \qquad\text{for every odd }\ell\ge3.
\end{equation}

Progress towards the matching upper bound came in several stages. The case $\ell=3$ is the celebrated Nash--Williams Conjecture~\cite{nashwilliams1970}, which predicts the threshold $3/4$. For longer odd cycles, Joos and K\"uhn~\cite{joos2022} showed that $\delta_{C_\ell}$ approaches $1/2$ as $\ell$ tends to infinity through odd integers. Bryant, Dukes, Horsley, Maenhaut and Montgomery~\cite{bryant2026} subsequently proved that for every  odd $\ell\ge5$, 
\[
 \delta_{C_\ell}\le\frac12+\frac1{2\ell-4}.
\]

Two recent breakthroughs brought the problem close to completion. In 2026, Delcourt and Postle~\cite{delcourt2026} considered triangle decompositions and resolved Nash--Williams' conjecture. 
\begin{theorem}[Delcourt--Postle~\cite{delcourt2026}]\label{thm:delcourt-postle-intro}
For a sufficiently large integer $n$,  every $n$-vertex triangle-divisible graph $G$  with minimum degree $\delta(G)\ge 3n/4$ 
admits a triangle decomposition.
\end{theorem}
Independently, Granet and Horsley~\cite{granet2026} proved that for  every odd $\ell \ge 73$, 
\[
 \delta_{C_\ell}=\frac12+\frac1{2\ell-2}. 
\]
The proof of Granet and Horsley~\cite{granet2026} relies on quantitative stability estimates for graphs close to the extremal bipartite construction. In particular, their bounds on the imbalance of the bipartition and the number of exceptional vertices require the cycle length to be sufficiently large in order to verify the Hall-type inequalities. Consequently, their argument establishes the optimal decomposition threshold only for odd cycles of length at least $73$, leaving the shorter odd cycles unresolved. In this paper, we resolve these remaining cases and thereby determine the decomposition thresholds for all odd cycles.

\begin{theorem}\label{thm:odd-cycle-threshold}
For every odd integer $\ell\ge3$, we have 
\[
 \delta_{C_\ell}=\frac12+\frac1{2\ell-2}.
\]
\end{theorem}

Our proof uses the condensation method of Bryant
et al.~\cite{bryant2026}, reducing odd-cycle
decompositions to fractional weighted-triangle
decompositions. The main new ingredient is a
rescaling of positive dual edge weights, which
allows us to apply the discharging estimate of
Delcourt and Postle~\cite{delcourt2026}
and obtain the sharp minimum degree bound.

The remainder of the paper is organised as follows.
Section~\ref{sec:notation} collects notation and the
fractional decomposition framework.
Section~\ref{sec:weighted-triangles} proves a sharp
fractional weighted triangle decomposition theorem.
Section~\ref{sec:odd-cycles} proves Theorem \ref{thm:odd-cycle-threshold}.

\section{Preliminaries}\label{sec:notation}

In this section, we introduce some notation and definitions.
Let $\R$ denote the set of real numbers and
$\R_{>0}$ the set of positive real numbers.
For a positive integer $k$, let $[k]:=\{1,\ldots,k\}$. For a finite set $S$, the space $\R^S$ consists of real vectors $x=(x_s)_{s\in S}$, and $\R_{\ge0}^S$ consists of those with $x_s\ge0$ for every $s\in S$.

For a graph $G$, let $V(G)$ and $E(G)$ be its vertex and edge sets, with \emph{order} $v(G):=|V(G)|$ and \emph{size} $e(G):=|E(G)|$.
For $v\in V(G)$, the \emph{degree} $d_G(v)$ of $v$ is the number of  edges incident with $v$. Let $\delta(G):=\min_{v\in V(G)}d_G(v)$ be the \emph{minimum degree} of $G$. A \emph{cut} $(A,B)$ of $G$ is a partition $V(G)=A\dotcuprel B$. An edge is a \emph{cross edge} of $(A,B)$ if its endpoints lie in different parts, and an \emph{internal edge} of $(A,B)$ if they lie in the same part. A graph is \emph{$r$-regular} if every vertex has degree $r$.  For a fixed graph $F$ with at least one edge, put $\gcd(F):=\gcd\{d_F(v):v\in V(F)\}$.
Recall that a graph $G$ is \emph{$F$-divisible} if $e(F)\mid e(G)$ and $\gcd(F)\mid d_G(v)$ for every $v\in V(G)$.

Fractional decompositions provide a natural linear
programming relaxation of graph decomposition problems
and play an important role in their study.
Haxell and R\"odl~\cite{haxell2001} established
a connection between fractional and approximate
decompositions.
This motivates the study of minimum-degree conditions
for fractional decompositions.
For cliques, such conditions were obtained by
Barber, K\"uhn, Lo, Montgomery and
Osthus~\cite{barber2017}, and subsequently improved
by Montgomery~\cite{montgomery2019}. We recall only the notions needed here.

A \emph{weighted graph} $W$ consists of an \emph{underlying graph} $F$ and a positive \emph{edge-weight function}
\[
 w_W:E(F)\longrightarrow \R_{>0}.
\]
An \emph{embedding} of $W$ into $G$ is an injective map $f:V(F)\to V(G)$ with $f(u)f(v)\in E(G)$ whenever $uv\in E(F)$. 
It defines a \emph{weighted copy} of $W$ in which each edge $f(uv):=f(u)f(v)$ has weight $w_W(uv)$. Let $f(E(F))=\{f(u)f(v): uv\in E(F)\}$. 
Write $\Emb(W,G)$ for the set of these embeddings and, for $e\in E(G)$, put
\[
 \Emb_e(W,G):=\{f\in\Emb(W,G):e\in f(E(F))\}.
\]
For $f\in\Emb_e(W,G)$, write $f^{-1}(e)$ for the unique edge $uv\in E(F)$ with $f(uv)=e$.

A \emph{fractional $W$-decomposition} of $G$ assigns each $f\in\Emb(W,G)$ a coefficient $x_f\ge0$ with
\[
 \sum_{f\in\Emb_e(W,G)}x_f\,w_W\bigl(f^{-1}(e)\bigr)=1
 \qquad\text{for every }e\in E(G).
\]
The copies may share edges, and coefficients of embeddings defining the same weighted copy may be added.
A \emph{fractional $F$-decomposition} is the special case in which every edge of $F$ has weight one.

Given a real number $a\ge1$, let $T_a=T_{a,1,1}$ denote a triangle with edge weights $a,1,1$; any of the three edges in a copy may carry weight $a$.  For a graph $G$, let $\cT(G)$ be the set of triangles in $G$ and put
\[
 \mathcal{I}(G):=\{(T,h):T\in\cT(G),\ h\in E(T)\}.
\]
Each pair $(T,h)\in\mathcal{I}(G)$ represents a copy of $T_a$ with edge-weight vector $w_{T,h}\in\R^{E(G)}$ given by
\begin{equation}\label{eq:weighted-triangle-vector}
 w_{T,h}(e):=
 \begin{cases}
 a,&e=h,\\
 1,&e\in E(T)\setminus\{h\},\\
 0,&e\notin E(T).
 \end{cases}
\end{equation} 
Note that $x\in \R_{\ge0}^{\mathcal{I}(G)}$ form a fractional $T_a$-decomposition exactly when
\begin{equation}\label{eq:fractional-Ta}
 \sum_{(T,h)\in\mathcal{I}(G)}x_{T,h}w_{T,h}(e)=1
 \qquad\text{for every }e\in E(G).
\end{equation}

Following the duality framework for fractional triangle
decompositions, we consider a \emph{signed edge weighting}
of $G$, that is, a function
\[
\phi:E(G)\longrightarrow\mathbb R.
\]
Such weightings play a central role in the work of
Delcourt and Postle~\cite{delcourt2026}.
We define the sets of positively and negatively weighted
edges by
\[
E_\phi^+:=\{e\in E(G):\phi(e)>0\},
\qquad
E_\phi^-:=\{e\in E(G):\phi(e)<0\},
\]
respectively.
Its \emph{positive mass} and \emph{negative mass} are
\[
 P(\phi):=\sum_{e\in E(G)}\max\{\phi(e),0\},
 \qquad
 N(\phi):=\sum_{e\in E(G)}\max\{-\phi(e),0\}.
\]
Clearly, 
\begin{equation}\label{eq:total-mass}
 \sum_{e\in E(G)}\phi(e)=P(\phi)-N(\phi).
\end{equation}
We call $\phi$ 
\emph{triangle-nonnegative} if
\[
 \sum_{e\in E(T)}\phi(e)\ge 0
 \qquad\text{for every }T\in\cT(G).
\]
Delcourt and Postle call such a weighting \emph{triangle-covered} \cite[Definition~4.3]{delcourt2026}.
Following their terminology \cite[Definition~4.6]{delcourt2026}, we call $\phi$ \emph{nonzero} if $\phi(e)\ne0$ for every $e\in E(G)$, and \emph{untied} if its edge weights are pairwise distinct.

We conclude this section by recalling the following
classical form of Farkas' lemma, which is a key
ingredient in our proof.

\begin{lemma}[Farkas' lemma {\cite[Chapter~7]{schrijver1986}}]
\label{thm:farkas}
Let $A\in\mathbb R^{m\times n}$ and
$b\in\mathbb R^m$.
Then exactly one of the following holds:
\begin{enumerate}
    \item There exists $x\in\mathbb R_{\ge0}^n$
    such that $Ax=b$.
    \item There exists $y\in\mathbb R^m$
    such that $A^{\mathsf T}y\ge0$ and
    $b^{\mathsf T}y<0$.
\end{enumerate}
\end{lemma}

\section{Fractional decompositions of weighted triangles}\label{sec:weighted-triangles}

Recall that, for $a\ge1$, $T_a$ is the triangle with edge weights $a,1,1$. A key ingredient in the approach of Bryant et al.~\cite{bryant2026} to odd-cycle decompositions is the following sufficient condition for fractional decompositions
into weighted triangles.

\begin{theorem}[Bryant et al.~{\cite[Theorem~1.2]{bryant2026}}]
\label{thm:bryant-weighted}
Let $a\ge 1$ be a real number, and let $G$ be an $n$-vertex graph. If
\[
\delta(G)\ge
\left(\frac12+\frac{1}{2a}\right)n,
\]
then $G$ admits a fractional $T_a$-decomposition.
\end{theorem}

The main contribution of this section is to improve the minimum-degree bound in
Theorem~\ref{thm:bryant-weighted} to the sharp value as follows. The proof relies on Farkas' lemma
and the discharging estimate of Delcourt and
Postle~\cite{delcourt2026}. Their result concerns
ordinary triangles, whereas we need to consider weighted
triangles.

\begin{theorem}\label{thm:weighted-triangle}
Let $a\ge1$ be a real number and put
\begin{equation}\label{eq:def-ca}
 c_a:=\frac12+\frac{1}{2(a+1)}.
\end{equation}
Every $n$-vertex graph  $G$ with $\delta(G)\ge c_a n$ has a fractional $T_a$-decomposition.
Moreover, the constant $c_a$ is best possible.
\end{theorem}

Fix a real number $a\ge1$ and a graph $G$. We call a signed edge weighting $y:E(G)\to\R$ \emph{dual-feasible} if, for every triangle of $G$ with edges $e_1,e_2,e_3$, we have 
\begin{equation}\label{eq:dual-feasible}
 \begin{aligned}
 ay(e_1)+y(e_2)+y(e_3)&\ge0,\\
 y(e_1)+ay(e_2)+y(e_3)&\ge0,\\
 y(e_1)+y(e_2)+ay(e_3)&\ge0.
 \end{aligned}
\end{equation}
Note that the three inequalities in \eqref{eq:dual-feasible} correspond to the three possible distinguished edges in a copy of $T_a$.

\begin{lemma}\label{lem:dual-criterion}
A graph $G$ has a fractional $T_a$-decomposition if and only if every dual-feasible signed edge weighting $y$ satisfies
\begin{equation}\label{eq:dual-target}
 \sum_{e\in E(G)}y(e)\ge0.
\end{equation}
\end{lemma}

\begin{proof}
Let $A$ be the $|E(G)|\times|\mathcal{I}(G)|$ matrix with rows indexed by $e\in E(G)$ and columns indexed by $(T,h)\in\mathcal{I}(G)$.
Set $A_{e,(T,h)}:=w_{T,h}(e)$, where $w_{T,h}$ is defined in \eqref{eq:weighted-triangle-vector}.
For $x\in\R^{\mathcal{I}(G)}$, the coordinate of $Ax$ indexed by $e$ is
\[
 (Ax)_e=\sum_{(T,h)\in\mathcal{I}(G)}x_{T,h}w_{T,h}(e)
 \qquad\text{for every }e\in E(G).
\]
The sum on the right is the total weighted contribution to edge $e$ from the copies with coefficients $x_{T,h}$.
Let $\1\in\R^{E(G)}$ be the all-ones vector.
By \eqref{eq:fractional-Ta}, $G$ has a fractional $T_a$-decomposition if and only if the system
\begin{equation}\label{eq:farkas-system}
 Ax=\1,
 \qquad
 x\in\R_{\ge0}^{\mathcal{I}(G)}
\end{equation}
has a solution.

By Lemma~\ref{thm:farkas} with $b=\1$, the system \eqref{eq:farkas-system} has no solution if and only if there exists a vector $y\in\R^{E(G)}$ with $A^{\mathsf T}y\ge0$ and $\1^{\mathsf T}y<0$.
We write $y(e)$ for the coordinate of $y$ indexed by $e\in E(G)$ and regard $y$ as a signed edge weighting of $G$.
For each $(T,h)\in\mathcal{I}(G)$, we have
\[
 (A^{\mathsf T}y)_{(T,h)}
 =\sum_{e\in E(G)}w_{T,h}(e)y(e)
 =a\,y(h)+\sum_{e\in E(T)\setminus\{h\}}y(e).
\]
This means that for each triangle $T$, the three choices of $h$ give precisely the three expressions on the left-hand sides of \eqref{eq:dual-feasible}. Thus,  $A^{\mathsf T}y\ge0$ is equivalent to $y$ being dual-feasible.
On the other hand, $\1^{\mathsf T}y<0$ implies $\sum_{e\in E(G)}y(e)<0$.

Therefore, 
Lemma~\ref{thm:farkas}  shows that \eqref{eq:farkas-system} has no solution if and only if $G$ admits a dual-feasible weighting with negative total edge weight.
Equivalently, $G$ has a fractional $T_a$-decomposition if and only if every dual-feasible weighting $y$ satisfies \eqref{eq:dual-target}.\qedhere
\end{proof}

We shall use a discharging estimate of Delcourt and Postle~\cite{delcourt2026}.
To state it, we first recall their demand function and linearized charge.
Let $G$ be a $n$-vertex graph, and let $\phi$ be a nonzero untied triangle-nonnegative signed edge weighting of $G$. Following Delcourt and Postle~\cite[Definition~4.11]{delcourt2026}, we define the \emph{demand} $D(e,f)$ of a negative edge $f\in E_\phi^-$ on a positive edge $e\in E_\phi^+$ as follows.
If no triangle contains both edges, set $D(e,f)=0$.
Otherwise, let $g$ be the third edge of their unique common triangle and put
\[
 p:=\phi(e)>0,
 \qquad
 s:=-\phi(f)>0,
 \qquad
 q:=\phi(g).
\]
Since $\phi$ is nonzero and untied, $q\ne0$ and $p\ne q$.
We set
\begin{equation}\label{eq:demand}
 D(e,f):=
 \begin{cases}
 s,&q<0,\\[2pt]
 \dfrac{sp}{p+q},&q>0\text{ and }\max\{p,q\}<s,\\[7pt]
 s,&q>0,\ p>q\text{ and }p\ge s,\\[2pt]
 0,&q>0,\ q>p\text{ and }q\ge s.
 \end{cases}
\end{equation}
For each $e\in E_\phi^+$, write
\begin{equation}\label{eq:total-demand}
 S(e):=\sum_{f\in E_\phi^-}D(e,f)
\end{equation}
for the \emph{total demand} on $e$. We define the \emph{linearized final charge}~\cite[Definitions~4.13--4.14]{delcourt2026} by
\begin{equation}
 Q_0(f):=\phi(f)+\sum_{e\in E_\phi^+}R_0(e,f)
 \qquad\text{for every }f\in E_\phi^-,\label{eq:linearized-final}
\end{equation}
where
\begin{equation}
 R_0(e,f):=\frac{4}{n}D(e,f)
 \left(1-\frac{S(e)}{n\phi(e)}\right)
 \qquad\text{for }e\in E_\phi^+,\ f\in E_\phi^-.\label{eq:linearized-transfer}
\end{equation}
We remark that Delcourt and Postle use the normalized demands $D(e,f)/\phi(e)$, whose sum over $f\in E_\phi^-$ is $S(e)/\phi(e)$.
Their sums are restricted to pairs of edges in a common triangle.
The same restriction leaves \eqref{eq:total-demand} and \eqref{eq:linearized-final} unchanged because $D(e,f)=0$ for every other pair.
In this notation, their estimate is as follows.

\begin{lemma}[Delcourt and Postle~{\cite[Lemma~10.4]{delcourt2026}}]\label{lem:dp-discharging}
Let $\varepsilon\in[0,1/4]$, and let $G$ be a $n$-vertex graph with $\delta(G)\ge(3/4-\varepsilon)n$.
Every nonzero untied triangle-nonnegative signed edge weighting $\phi$ of $G$ satisfies
\[
 Q_0(f)\ge-4\varepsilon|\phi(f)|
 \qquad\text{for every }f\in E_\phi^-.
\]
\end{lemma}

Using this estimate, we obtain the following bound on the positive and negative masses of a triangle-nonnegative weighting.

\begin{lemma}\label{lem:mass-estimate}
Let $a\ge1$ be a real number, and let $G$ be a $n$-vertex graph of with $\delta(G)\ge c_a n$, where $c_a$ is defined in \eqref{eq:def-ca}.
Then every triangle-nonnegative signed edge weighting $\phi$ of $G$ satisfies
\begin{equation}\label{eq:mass-estimate}
 P(\phi)\ge\frac{2}{a+1}N(\phi).
\end{equation}
\end{lemma}

\begin{proof}
Let $G$ be a $n$-vertex graph of with $\delta(G)\ge c_a n$, and let $\phi$ be a triangle-nonnegative signed edge weighting of $G$.  Set 
\[
 \rho:=\frac34-c_a=\frac{a-1}{4(a+1)}.
\]
Using \eqref{eq:total-mass} and $1-4\rho=2/(a+1)$, to prove \eqref{eq:mass-estimate}, it suffices to show 
\begin{equation}\label{eq:charge-target}
 \sum_{e\in E(G)}\phi(e)\ge-4\rho N(\phi).
\end{equation}

Suppose first that $\phi$ is nonzero and untied.
Give each edge $e$ the initial charge $\phi(e)$.
To prove \eqref{eq:charge-target}, we use the classical discharging method by  redistributing $\phi(e)$ from positive edges to negative edges while preserving the total charge. After discharging, the  final charges that are nonnegative on positive edges and at least $-4\rho|\phi(f)|$ on each negative edge $f$.

The discharging rule is defined as follows: for $e\in E_\phi^+$ and $f\in E_\phi^-$, define the transfer from $e$ to $f$ by
\begin{equation}\label{eq:redistribution}
 R(e,f):=
 \begin{cases}
 \displaystyle \phi(e)\frac{D(e,f)}{S(e)},&S(e)>0,\\[7pt]
 0,&S(e)=0.
 \end{cases}
\end{equation}
The final charges are therefore
\begin{align}
 Q(e)&:=\phi(e)-\sum_{f\in E_\phi^-}R(e,f)
 \qquad\text{for every }e\in E_\phi^+,\label{eq:final-positive}\\*
 Q(f)&:=\phi(f)+\sum_{e\in E_\phi^+}R(e,f)
 \qquad\text{for every }f\in E_\phi^-.\label{eq:final-negative}
\end{align}
Summing \eqref{eq:final-positive} and \eqref{eq:final-negative} over $E_\phi^+$ and $E_\phi^-$, respectively, gives that 
\begin{equation}\label{eq:charge-conservation}
 \sum_{h\in E(G)}Q(h)=\sum_{h\in E(G)}\phi(h).
\end{equation}

We first show that every positive edge has nonnegative final charge. For $e\in E_\phi^+$ with $S(e)>0$, substituting \eqref{eq:redistribution} into \eqref{eq:final-positive} gives
\[
 Q(e)=\phi(e)-\frac{\phi(e)}{S(e)}\sum_{f\in E_\phi^-}D(e,f)=0.
\]
If $S(e)=0$, then no charge is transferred from $e$, so $Q(e)=\phi(e)>0$.
Thus
\begin{equation}\label{eq:positive-charge}
 Q(e)\ge0\qquad\text{for every }e\in E_\phi^+.
\end{equation}

We next bound the charges of negative edges.
It follows from $a\ge1$ that  $0\le\rho<1/4$, and then $\delta(G)\ge c_a n=(3/4-\rho)n$.
Applying Lemma~\ref{lem:dp-discharging} to $(G,\phi)$ with $\varepsilon=\rho$ gives
\begin{equation}\label{EQ:Q-0(f)}
  Q_0(f)\ge-4\rho|\phi(f)|
 \qquad\text{for every }f\in E_\phi^-.   
\end{equation}
To obtain the same lower bound for $Q(f)$, we show that 
\begin{equation}\label{EQ:Q(f)>=Q0}
 Q(f)\ge Q_0(f)
 \qquad\text{for every }f\in E_\phi^-.   
\end{equation}
By \eqref{eq:final-negative} and \eqref{eq:linearized-final}, it suffices to prove 
\[
 R(e,f)\ge R_0(e,f)
 \qquad\text{for every }e\in E_\phi^+\text{ and }f\in E_\phi^-.
\]
Fix $e\in E_\phi^+$ and $f\in E_\phi^-$.
If $S(e)=0$, then  $D(e,f)=0$ and so $R(e,f)=R_0(e,f)=0$. We are done. 
Suppose that $S(e)>0$. Set $x_e:=S(e)/(n\phi(e))>0$.
Substituting $x_e$ into \eqref{eq:redistribution} and \eqref{eq:linearized-transfer} gives
\[
 R(e,f)-R_0(e,f)
 =\frac{D(e,f)}{n}\left(\frac{1}{x_e}-4+4x_e\right)
 =\frac{D(e,f)}{n}\frac{(2x_e-1)^2}{x_e}\ge0.
\]

Combining \eqref{EQ:Q-0(f)} and \eqref{EQ:Q(f)>=Q0}, we have 
\begin{equation}\label{eq:negative-charge}
 Q(f)\ge Q_0(f)\ge-4\rho|\phi(f)|
 \qquad\text{for every }f\in E_\phi^-.
\end{equation}
By \eqref{eq:charge-conservation}, \eqref{eq:positive-charge} and \eqref{eq:negative-charge}, we obtain
\[
 \sum_{e\in E(G)}\phi(e)
 =\sum_{h\in E(G)}Q(h)
 =\sum_{e\in E_\phi^+}Q(e)+\sum_{f\in E_\phi^-}Q(f)
 \ge-4\rho\sum_{f\in E_\phi^-}|\phi(f)|
 =-4\rho N(\phi).
\]
This proves \eqref{eq:charge-target}, and hence \eqref{eq:mass-estimate}, when $\phi$ is nonzero and untied.\\

For a general triangle-nonnegative weighting $\phi$, write $E(G)=\{e_1,\ldots,e_m\}$, where $m=|E(G)|$.
For $t>0$, define
\[
 \phi_t(e_i):=\phi(e_i)+i\,t
 \qquad\text{for every }i\in\{1,\ldots,m\}.
\]
For a triangle $T$ of $G$, we have 
\[
 \sum_{e_i\in E(T)}\phi_t(e_i)
 =\sum_{e_i\in E(T)}\phi(e_i)
 +t\sum_{e_i\in E(T)}i
 \ge0.
\]
This means $\phi_t$ remains triangle-nonnegative for any $t>0$.
A zero weight can occur only at $t=-\phi(e_i)/i$, and equal weights on distinct edges $e_i,e_j$ can occur only at $t=(\phi(e_j)-\phi(e_i))/(i-j)$.
These conditions exclude only finitely many values of $t$.
Choose a sequence of positive numbers $t_k\to0$ as $k\to \infty$ avoiding those values.
Each $\phi_{t_k}$ is then nonzero, untied and triangle-nonnegative. The above argument shows 
\begin{equation}\label{eq:perturbed-mass}
 P(\phi_{t_k})\ge\frac{2}{a+1}N(\phi_{t_k})
 \qquad\text{for every }k\ge1.
\end{equation}
Since $P$ and $N$ are finite sums of continuous functions, taking limits in \eqref{eq:perturbed-mass} yields \eqref{eq:mass-estimate}.
\end{proof}

We now prove Theorem~\ref{thm:weighted-triangle}.

\begin{proof}[Proof of Theorem~\ref{thm:weighted-triangle}]
Let $a\ge1$, and let $G$ be a $n$-vertex graph  with $\delta(G)\ge c_a n$.
We first show that $G$ has a fractional $T_a$-decomposition.
By Lemma~\ref{lem:dual-criterion}, it suffices to show that $\sum_{e\in E(G)}y(e)\ge0$
for every dual-feasible weighting $y$ of $G$.
By \eqref{eq:total-mass}, this is equivalent to
\[
P(y)\ge N(y).
\]

Fix a dual-feasible weighting $y$ of $G$.
We first rescale the nonnegative weights of $y$ and show that the resulting weighting $\phi$ is triangle-nonnegative.
Set
\begin{equation}\label{eq:rescaling-definition}
 \beta:=\frac{a+1}{2},
 \qquad
 \phi(e):=
 \begin{cases}
 y(e)/\beta,&y(e)\ge0,\\
 y(e),&y(e)<0.
 \end{cases}
\end{equation}
Then $1\le\beta\le a$, and 
\begin{equation}\label{eq:rescaled-conclusions}
 P(y)=\beta P(\phi),
 \qquad
 N(y)=N(\phi).
\end{equation}

We claim that $\sum_{e\in E(T)}\phi(e)\ge0$ for each triangle $T$ of $G$. Indeed, let $T$ be a triangle of $G$ and label its edges so that $y(e_1)\le y(e_2)\le y(e_3)$. 
By \eqref{eq:dual-feasible}, we have $y(e_3)\ge0$.
If $y(e_1)\ge0$, then $\phi(e)=y(e)/\beta\ge0$ for each $e\in E(T)$, so $\sum_{e\in E(T)}\phi(e)\ge0$. Suppose next that $y(e_1)<0\le y(e_2)$.
The first inequality in \eqref{eq:dual-feasible} gives $y(e_2)+y(e_3)\ge-ay(e_1)$.
It follows from  $\beta\le a$ that 
\[
 \sum_{e\in E(T)}\phi(e)
 =y(e_1)+\frac{y(e_2)+y(e_3)}{\beta}
 \ge\left(1-\frac{a}{\beta}\right)y(e_1)
 \ge0.
\]
At last, suppose that $y(e_1)\le y(e_2)<0\le y(e_3)$.
Adding the first two inequalities in \eqref{eq:dual-feasible} and dividing by $2\beta=a+1$, we obtain
\[
 \sum_{e\in E(T)}\phi(e)
 =y(e_1)+y(e_2)+\frac{y(e_3)}{\beta}\ge0.
\]

The above argument shows that $\phi$ is triangle-nonnegative. Applying Lemma~\ref{lem:mass-estimate} to $\phi$ and using \eqref{eq:rescaled-conclusions} with $\beta=(a+1)/2$, we obtain
\[
 P(y)=\beta P(\phi)
 \ge\beta\frac{2}{a+1}N(\phi)
 =N(\phi)=N(y).
\]
Thus, Lemma~\ref{lem:dual-criterion} gives a fractional $T_a$-decomposition of $G$.  \\

To prove that $c_a$ is best possible, we first derive a cut inequality that every graph with a fractional $T_a$-decomposition must satisfy. Suppose that a graph $H$ has a fractional $T_a$-decomposition with coefficients $x_{T,h}\ge0$ for each $(T,h)\in\mathcal{I}(H)$.
Fix a cut $(A,B)$ of $H$. 
Let $E_{\mathrm{int}}$ denote the set of internal edges of $(A,B)$, and let $e_{\mathrm{cross}}$ and $e_{\mathrm{int}}$ denote the numbers of cross edges and internal edges of $(A,B)$, respectively. 
Write $\sum_{T,h}$ for $\sum_{(T,h)\in\mathcal{I}(H)}$. Since every triangle contains an internal edge
and all edge weights of $T_a$ are at least $1$,
we have $\sum_{e\in E_{\mathrm{int}}}w_{T,h}(e)\ge1$
for every $(T,h)\in\mathcal I(G)$. Thus, summing \eqref{eq:fractional-Ta} over $E_{\mathrm{int}}$ and using $x_{T,h}\ge0$ gives
\begin{equation}\label{EQ:e-int}
  e_{\mathrm{int}}
 =\sum_{e\in E_{\mathrm{int}}}1
 =\sum_{e\in E_{\mathrm{int}}}\sum_{T,h}x_{T,h}w_{T,h}(e)
 =\sum_{T,h}x_{T,h}
   \sum_{e\in E_{\mathrm{int}}}w_{T,h}(e)
 \ge\sum_{T,h}x_{T,h}.   
\end{equation}
On the other hand, recall that each copy of $T_a$  has total edge weight $a+2$.  Summing \eqref{eq:fractional-Ta} over $E(H)$ gives
\begin{equation}\label{e-cross-e-int}
  e_{\mathrm{cross}}+e_{\mathrm{int}}
 =\sum_{e\in E(H)}1
 =\sum_{e\in E(H)}\sum_{T,h}x_{T,h}w_{T,h}(e)
 =\sum_{T,h}x_{T,h}
   \sum_{e\in E(H)}w_{T,h}(e)
 =(a+2)\sum_{T,h}x_{T,h}.   
\end{equation}
Combining \eqref{EQ:e-int} and \eqref{e-cross-e-int}, we obtain 
\begin{equation}\label{eq:weighted-cut-global}
 e_{\mathrm{cross}}\le(a+1)e_{\mathrm{int}}.
\end{equation}

We now construct graphs with no fractional $T_a$-decomposition and minimum-degree ratios tending to $c_a$.
Let $m$ be an even integer, and let $r_m$ be the largest nonnegative even integer strictly smaller than $m/(a+1)$.
For all sufficiently large $m$,
\begin{equation}\label{eq:rm-bounds}
 \frac{m}{a+1}-2\le r_m<\frac{m}{a+1}.
\end{equation}
Let $G_m$ be the graph on two disjoint sets $A$ and $B$ of size $m$ obtained by adding all cross edges of $(A, B)$ and placing an $r_m$-regular graph inside each part.
Then we have 
\begin{equation}\label{eq:weighted-construction-data}
 v(G_m)=2m,
 \qquad
 \delta(G_m)=m+r_m,
 \qquad
 e_{\mathrm{cross}}=m^2,
 \qquad
 e_{\mathrm{int}}=mr_m.
\end{equation}
By \eqref{eq:def-ca}, \eqref{eq:rm-bounds} and \eqref{eq:weighted-construction-data}, the minimum-degree ratios of $G_m$ satisfy
\[
 \frac{\delta(G_m)}{v(G_m)}
 =\frac12+\frac{r_m}{2m}<c_a,
 \qquad
 \lim_{m\to\infty}\frac{\delta(G_m)}{v(G_m)}=c_a.
\]
Moreover, by \eqref{eq:rm-bounds} and \eqref{eq:weighted-construction-data}, we have
\[
e_{\mathrm{cross}}=m^2>(a+1)mr_m=(a+1)e_{\mathrm{int}},
\]
which implies that $G_m$ has no fractional $T_a$-decomposition by \eqref{eq:weighted-cut-global}. This completes the proof of Theorem \ref{thm:weighted-triangle}.

\end{proof}

\section{Proof of Theorem~\ref{thm:odd-cycle-threshold}}\label{sec:odd-cycles}
In this section, we prove Theorem~\ref{thm:odd-cycle-threshold}. 
For a positive integer $k$, a \emph{proper $k$-colouring} of a graph $F$ is a map $\kappa:V(F)\to[k]$ assigning different colours to the endpoints of every edge.
For $i\in[k]$, let $V_i:=\kappa^{-1}(i)$ be the \emph{colour class} of colour $i$.
The \emph{condensation} $W$ associated with $\kappa$ has vertex set $[k]$, with vertex $i$ representing the colour class $V_i$.
Distinct vertices $i$ and $j$ are adjacent in $W$ if and only if $F$ has an edge between $V_i$ and $V_j$.
Each edge $ij$ of $W$ has weight equal to the number of edges of $F$ between $V_i$ and $V_j$ \cite[Section~1]{bryant2026}.

Let $F$ be a fixed graph and let $\eta>0$.
An \emph{$\eta$-approximate $F$-decomposition} of an $n$-vertex graph $G$ is a collection of edge-disjoint copies of $F$ leaving at most $\eta n^2$ edges of $G$ uncovered.
Define the \emph{$\eta$-approximate decomposition threshold} $\delta_F^\eta$ to be the infimum of all real numbers $d\ge 0$ such that every sufficiently large graph $G$ with
$\delta(G)\ge d\,v(G)$ has an $\eta$-approximate $F$-decomposition.
The definition of $\delta_F^\eta$ imposes no divisibility condition.

The following theorem of Bryant et al. \cite{bryant2026} shows that a fractional $W$-decomposition yields an approximate $F$-decomposition whenever $W$ is a condensation of $F$.

\begin{theorem}[Bryant et~al.\ {\cite[Theorem~1.3]{bryant2026}}]\label{thm:condensation}
Let $F$ be a fixed graph and $W$ a condensation of $F$.
For every $\eta>0$, there exists an integer $n_0=n_0(F,\eta)$ such that every graph $G$ of order $n\ge n_0$ with a fractional $W$-decomposition has an $\eta$-approximate $F$-decomposition.
\end{theorem}
For odd cycles, the following theorem of Barber et al. \cite{barber2016} gives exact decompositions with an arbitrarily small increase in the minimum-degree bound for approximate decompositions.

\begin{theorem}[Barber et~al.\ {\cite[Theorem~1.4]{barber2016}}]\label{thm:absorption}
Fix an odd integer $\ell\ge3$.
For every $\gamma>0$, there exist $\eta>0$ and an integer $n_0$ such that every $C_\ell$-divisible graph $G$ of order $n\ge n_0$ with
\[
 \delta(G)\ge\bigl(\delta_{C_\ell}^\eta+\gamma\bigr)n
\]
has a $C_\ell$-decomposition.
\end{theorem}

Combining Theorem~\ref{thm:weighted-triangle} with the two results above, we prove Theorem~\ref{thm:odd-cycle-threshold}.

\begin{proof}[Proof of Theorem~\ref{thm:odd-cycle-threshold}]
Fix an odd integer $\ell\ge3$, and put $c_\ell:=\ell/(2\ell-2)$. By \eqref{eq:odd-cycle-lower-intro}, it suffices to show $ \delta_{C_\ell}\le c_\ell$. Together with Theorem~\ref{thm:absorption}, we only need to prove 
\begin{equation}\label{eq:approximate-threshold}
 \delta_{C_\ell}^\eta\le c_\ell
 \qquad\text{for every }\eta>0.
\end{equation}
Fix $\eta>0$.
We prove \eqref{eq:approximate-threshold} by showing that for sufficiently large $m$, every $m$-vertex graph $H$ with $\delta(H)\ge c_\ell m$ has an $\eta$-approximate $C_\ell$-decomposition.

Take $W:=T_{\ell-2,1,1}$. We claim that $W$ is a condensation of $C_\ell$. Indeed, let $v_1,\ldots,v_\ell$ be the vertices of $C_\ell$ in cyclic order.
We define a proper $3$-colouring by giving $v_\ell$ colour $3$ and alternating colours $1$ and $2$ along the path $v_1\cdots v_{\ell-1}$, starting with colour $1$ at $v_1$.
Let $e_{ij}$ denote the number of edges between distinct colour classes $i$ and $j$. Then 
\[
 e_{12}=\ell-2,
 \qquad
 e_{13}=1,
 \qquad
 e_{23}=1.
\]
Hence $W$ is a condensation of $C_\ell$.

Setting $a:=\ell-2\ge1$ in \eqref{eq:def-ca} gives
\[
 c_a
 =\frac12+\frac{1}{2(a+1)}
 =\frac12+\frac{1}{2(\ell-1)}
 =c_\ell.
\]
Recall that $a\ge1$, $W=T_{a,1,1}$ and $\delta(H)\ge c_\ell m=c_a m$. Theorem~\ref{thm:weighted-triangle} gives a fractional $W$-decomposition of $H$.  By Theorem~\ref{thm:condensation}, $H$ has an $\eta$-approximate $C_\ell$-decomposition, and then $\delta_{C_\ell}^\eta\le c_\ell$.
Since $\eta>0$ was arbitrary, this proves \eqref{eq:approximate-threshold} and hence Theorem~\ref{thm:odd-cycle-threshold}. 
\end{proof}

\section{Concluding remarks}

The main ingredient of our proof is a sharp fractional
decomposition theorem for weighted triangles (Theorem \ref{thm:weighted-triangle}).
Although we use the discharging estimate of Delcourt
and Postle~\cite{delcourt2026}, their result applies
to triangle-nonnegative edge weightings, whereas
weighted-triangle decompositions lead to different
dual constraints. Our contribution is a rescaling
of the positive edge weights that preserves the
negative weights and converts these constraints
into triangle nonnegativity. 

Theorem \ref{thm:weighted-triangle} holds for every real $a\ge 1$, not only for the values
$a=\ell-2$ arising from odd-cycle decompositions. Thus, the result is of independent interest
beyond its application to odd cycles. It is therefore
natural to ask whether the sharp minimum-degree
threshold can also be determined for fractional
decompositions into triangles with arbitrary positive
edge weights $(a,b,c)$.

\section*{Acknowledgments}
J.H.\ was supported by the National Key R\&D Program of China (No.\ 2023YFA1010202).

\section*{Declaration on the use of generative AI}
During the preparation of this manuscript, the authors used  generative artificial intelligence. 
The principal ideas and proof outline had already been developed by the authors before the tool was used. 
The tool assisted with a targeted literature search, most notably by helping to locate the quantitative discharging estimate of Delcourt and Postle  \cite[Lemma~10.4]{delcourt2026}, restated here as Lemma~\ref{lem:dp-discharging}, and provided assistance with manuscript preparation.  The authors reviewed and edited all AI-assisted material and take full responsibility for the content of the manuscript.
\setlength{\bibsep}{4pt}
\bibliographystyle{plain}
\bibliography{references}

@article{barber2017,
  author  = {Barber, B. and K{\"u}hn, D. and Lo, A. and Montgomery, R. and Osthus, D.},
  title   = {Fractional clique decompositions of dense graphs and hypergraphs},
  journal = {J. Combin. Theory Ser. B},
  volume  = {127},
  pages   = {148--186},
  year    = {2017}
}

@article{barber2016,
  author  = {Barber, B. and K{\"u}hn, D. and Lo, A. and Osthus, D.},
  title   = {Edge-decompositions of graphs with high minimum degree},
  journal = {Adv. Math.},
  volume  = {288},
  pages   = {337--385},
  year    = {2016}
}

@article{bryant2026,
  author  = {Bryant, D. and Dukes, P. and Horsley, D. and Maenhaut, B. and Montgomery, R.},
  title   = {On decomposition thresholds for odd-length cycles and other tripartite graphs},
  journal = {J. Combin. Theory Ser. B},
  volume  = {180},
  pages   = {94--110},
  year    = {2026}
}

@article{bryant2014,
  author  = {Bryant, D. and Horsley, D. and Pettersson, W.},
  title   = {Cycle decompositions {V}: complete graphs into cycles of arbitrary lengths},
  journal = {Proc. Lond. Math. Soc. (3)},
  volume  = {108},
  number  = {5},
  pages   = {1153--1192},
  year    = {2014}
}

@article{delcourt2021,
  author  = {Delcourt, M. and Postle, L.},
  title   = {Progress towards {Nash-Williams'} conjecture on triangle decompositions},
  journal = {J. Combin. Theory Ser. B},
  volume  = {146},
  pages   = {382--416},
  year    = {2021}
}

@unpublished{delcourt2026,
  author        = {Delcourt, M. and Postle, L.},
  title         = {A proof of {Nash-Williams'} conjecture},
  note          = {Preprint, arXiv:\allowbreak 2606.11178v1},
  eprint        = {2606.11178v1},
  archivePrefix = {arXiv},
  year          = {2026}
}

@article{glock2019,
  author  = {Glock, S. and K{\"u}hn, D. and Lo, A. and Montgomery, R. and Osthus, D.},
  title   = {On the decomposition threshold of a given graph},
  journal = {J. Combin. Theory Ser. B},
  volume  = {139},
  pages   = {47--127},
  year    = {2019}
}

@article{glock2023,
  author  = {Glock, S. and K{\"u}hn, D. and Lo, A. and Osthus, D.},
  title   = {The existence of designs via iterative absorption: hypergraph {$F$}-designs for arbitrary {$F$}},
  journal = {Mem. Amer. Math. Soc.},
  volume  = {284},
  number  = {1406},
  year    = {2023}
}

@unpublished{granet2026,
  author        = {Granet, B. and Horsley, D.},
  title         = {Determining decomposition thresholds for long odd cycles},
  note          = {Preprint, arXiv:\allowbreak 2606.21548v1},
  eprint        = {2606.21548v1},
  archivePrefix = {arXiv},
  year          = {2026}
}

@article{haxell2001,
  author  = {Haxell, P. E. and R{\"o}dl, V.},
  title   = {Integer and fractional packings in dense graphs},
  journal = {Combinatorica},
  volume  = {21},
  number  = {1},
  pages   = {13--38},
  year    = {2001}
}

@article{joos2022,
  author  = {Joos, F. and K{\"u}hn, M.},
  title   = {Fractional cycle decompositions in hypergraphs},
  journal = {Random Structures Algorithms},
  volume  = {61},
  number  = {3},
  pages   = {425--443},
  year    = {2022}
}

@article{keevashdesigns,
  author  = {Keevash, P.},
  title   = {The existence of designs},
  journal = {Ann. of Math. (2)},
  year    = {to appear}
}

@article{montgomery2019,
  author  = {Montgomery, R.},
  title   = {Fractional clique decompositions of dense graphs},
  journal = {Random Structures Algorithms},
  volume  = {54},
  number  = {4},
  pages   = {779--796},
  year    = {2019}
}

@incollection{nashwilliams1970,
  author    = {Nash-Williams, C. {\relax St}. J. A.},
  title     = {An unsolved problem concerning decomposition of graphs into triangles},
  booktitle = {Combinatorial Theory and its Applications {III}},
  editor    = {Erd\H{o}s, P. and R\'enyi, A. and S\'os, V. T.},
  publisher = {North-Holland},
  address   = {Amsterdam},
  pages     = {1179--1183},
  year      = {1970}
}

@book{schrijver1986,
  author    = {Schrijver, A.},
  title     = {Theory of Linear and Integer Programming},
  publisher = {Wiley},
  address   = {Chichester},
  year      = {1986}
}

@article{taylor2019,
  author  = {Taylor, A.},
  title   = {On the exact decomposition threshold for even cycles},
  journal = {J. Graph Theory},
  volume  = {90},
  number  = {3},
  pages   = {231--266},
  year    = {2019}
}

@article{Wilson1975,
  author  = {Wilson, R. M.},
  title   = {Decompositions of complete graphs into subgraphs isomorphic to a given graph},
  journal = {Congr. Numer.},
  volume  = {15},
  pages   = {647--659},
  year    = {1975}
}

\end{document}